\documentclass[12pt]{article}
\usepackage{geometry} 
\usepackage[ utf8 ]{inputenc}
\usepackage[T1]{fontenc}
\usepackage[english,french]{babel}
\usepackage{amsmath}
\usepackage{amsfonts}
\usepackage{amssymb}
\usepackage{amsthm}
\usepackage{textcomp}
\usepackage{eucal}
\usepackage{array}
\usepackage{pb-diagram}
\theoremstyle{plain}
\newtheorem{theorem}{Théorème}
\newtheorem{corollary}{Corollaire}

\newtheorem{proposition}{Proposition}
\newtheorem{notation}{Notation}
\theoremstyle{definition}

\theoremstyle{remark}
\newtheorem{remark}{Remarque}
\date{}

\title{ Représentations de réflexion de groupes de Coxeter\\Quatrième partie: La représentation $R$ est réductible. Généralités}
\author{François ZARA
}
\begin{document}
\maketitle
\begin{abstract}
Dans cette quatrième partie, (avec les notations des parties précédentes) on fait les hypothèses suivantes: $(W,S)$ est un système de Coxeter irréductible, $2$-sphérique et $S$ est fini. Soit $R:W\to GL(M)$ une représentation de réflexion réductible de $W$. On pose $G:= Im R$. Chaque sous-espace de $M$ $(\neq M)$ fixé par $G$ est contenu dans $C_{M}(G)$. On pose $M':=M/C_{M}(G)$ et $N(G):=\{g|g\in G, g \,\text{fixe}\, M'\}$. On appelle $N(G)$ le sous-groupe des translations de $G$. Un des buts de cette partie est d'étudier $M'$ et $N(G)$.
\end{abstract}
\begin{otherlanguage}{english}
\begin{abstract}
In this fourth part, (with the notations of the preceding parts) we make the following hypothesis:  $(W,S)$ is a  Coxeter system, irreducible, $2$-spherical and $S$ is finite. Let $R:W\to GL(M)$ be a reducible reflection representation of $W$. Let $G:= Im\,R$. Each sub-space of $M$ $(\neq M)$ stabilize by $G$ is contained in $C_{M}(G)$. Let $M':=M/C_{M}(G)$ and $N(G):=\{g|g\in G,g\, \text{acts trivially on}\,M'$. We call $N(G)$ the translation sub-group of $G$. One of the goals of this part is to study $M'$ and $N(G)$.
\end{abstract}
\end{otherlanguage}
\let\thefootnote\relax\footnote{Mots clés et phrases: groupes de Coxeter, groupes de réflexion.Représentation de réflexion réductible.}
\let\thefootnote\relax\footnote{Mathematics Subject Classification. 20F55,22E40,51F15,33C45.}\section{Introduction}
Dans toute cette partie on fait les hypothèses suivantes: $(W,S)$ est un système de Coxeter satisfaisant aux conditions H(Cox) et $R:W\to GL(M)$ est une représentation de réflexion réductible. On pose $G:=Im R$. Le but est d'étudier la structure de $G$.
\subsection{Le théorème fondamental}
\begin{theorem}
Soit $(W,S)$ un système de Coxeter satisfaisant aux hypothèses H(Cox) et soit $R:W \to GL(M)$ une représentation de réflexion de $W$ obtenue par la construction fondamentale. On pose $G:=Im\, R$. Alors:
\begin{enumerate}
  \item $H:=\bigcap_{s\in S}H(s)$ est le plus grand sous-espace de $M$ stable par $G$.
  \item La suite $(\ast)$ de $KG$-modules, où $M':=M/\bigcap_{s\in S}H(s)$ est exacte et non scindée:
  \[
   (\ast) \qquad {0}\to  \bigcap_{s\in S}H(s) \to M \xrightarrow{\pi_{M}} M' \to {0}
  \]
  où $\pi_{M}$ est la projection canonique.
  \item La représentation $R$ est irréductible si et seulement si $H=\{0\}$, si et seulement si $\Delta(G)\neq\{0\}$.
  \item On suppose que $R$ est réductible. Alors $G$ opère sur $M'$ et l'on a la suite exacte $(\ast\ast)$:
  \[
 (\ast\ast) \qquad  \{1\}\to N(G)\to G \xrightarrow{\pi} G' \to \{1\}
  \]
  où $N(G):=\{g|\in G, g\; \text{opère trivialement sur}\, M'\}=\ker \pi$, $G':=G/N(G)$ et $\pi$, la projection canonique, est un bon morphisme.\\
  De plus $M'$ est un $KG$-module simple et un $KG'$-module simple. $G'$ opère comme un groupe de réflexion sur $M'$.
  \item $N(G)$ est un groupe commutatif sans torsion et tous ses éléments non triviaux sont des applications unipotentes.
\end{enumerate}
\end{theorem}
\begin{proof}
1) Soit $V$ un sous-espace de $M$ stable par $G$. Si $V \, \nsubseteq \bigcap_{s\in S}H(s)$ alors $\exists v \in V, \exists s \in S$ tels que $s(v) \in V$ et $s(v)\neq v$. Dans ces conditions $s(v)-v=\lambda a_{s}$ ($\lambda \in K^{*}$) est dans $V$ et $a_{s} \in V$. Comme $\Gamma(G)$ est connexe, en appliquant les différents éléments de $S$ à $a_{s}$, nous voyons que $\forall t \in S$, $a_{t}\in V$. Il en résulte que $V=M$. Il est clair que $H\subset C_{M}(G)$ et que $C_{M}(G)$ est stable par $G$, donc $H=C_{M}(G)$.\\
Les 2), 3)  sont clairs.\\
4) Nous montrons que $\pi$ est un bon morphisme. Soient $s$ et $t$ deux éléments de $S$. par hypothèse $st$ est d'ordre fini $n$ car $W$ est $2$-sphérique. Alors $\pi(st)$ est d'ordre $n'$ où $n'$ est un diviseur de $n$ et $\pi(st)^{n'}$ est d'ordre $1$ donc $(st)^{n'}\in N(G)$, mais tous les éléments non triviaux de $N(G)$ sont d'ordre infini et $(st)^{n'}$ est d'ordre fini, donc $(st)^{n'}=1$ et $n'=1$, d'où le résultat.\\
5) Comme les éléments de $N(G)$ opèrent trivialement sur $H$ et $M/H$, toutes leurs valeurs propres sont égales à $1$, donc ce sont des applications unipotentes. On voit aussi que $N(G)$ est commutatif et sans torsion.
\end{proof}
\begin{corollary}
Si $G'$ est un groupe de Coxeter $2$-sphérique alors la suite $(\star\star)$ est scindée.
\end{corollary}
\begin{proof}

\end{proof}
\begin{notation}
Avec les hypothèses du théorème précédent, si $R$ est réductible, on dit que $R$ est une représentation de réflexion \textbf{affine} de $W$, que $G$ est un groupe de réflexion \textbf{affine} et que $N(G)$ est son sous-groupe des \textbf{translations}.
\end{notation}
Le problème ici est que l'on ne sait pas si $N(G)$ est non trivial ou non.
\subsection{Des bases adaptées}
Dans toute cette section $(W,S)$ est un système de Coxeter qui satisfait aux conditions H(Cox). Soit $R: W\to GL(M)$ une représentation de réflexion \textbf{réductible} de $W$. On garde les notations précédentes et on utilise la construction fondamentale.\\
Pour pouvoir faire des calculs explicites, on va choisir une base de $H$ et uns base de $M$ qui ont de bonnes propriétés. On suppose que $n_{0}:= \dim H\geqslant 1$ et l'on pose $n_{1}:=n-n_{0}$.
\begin{remark}
On a toujours $n_{0}\leqslant n_{1}$.
\end{remark}
\begin{proof}
En effet si l'on avait $n_{1}<n_{0}$, le groupe $G'=\pi(G)$ n'opérerait pas irréductiblement sur $M'(=M/H)$.
\end{proof}
Il en résulte qu'il existe $S_{1} \subset S$, $|S_{1}|=n_{1}$ tel que $\Delta(<S_{1}>)\neq 0$.
\begin{notation}
On pose $S_{0}:=S-S_{1}$ et $G_{i}:=<S_{i}>(i=0,1)$. On pose $S_{1}:=\{s_{i_{1}},s_{i_{2}},\cdots,s_{i_{n_{1}}}\}$ et on suppose que $i_{1}<i_{2}<\cdots<i_{n_{1}}$.
\end{notation}
Il est clair que $\pi(G_{1})\simeq G_{1}$.\\
Soit $M_{1}:=<a_{i_{1}},a_{i_{2}},\cdots,a_{i_{n_{1}}}>$. On a $M=H\oplus M_{1}$. Soit $\pi_{1}:M\to M/M_{1}$ la projection canonique. On a $\pi_{1}(M)=\pi_{1}(H)$. \textbf{On choisit} pour $i\in\{1,2,\cdots,n\}-\{i_{1},i_{2},\cdots,i_{n_{1}}\}$, $b_{i}$ dans $H$ de telle sorte que l'on ait $\pi_{1}(b_{i})=\pi_{1}(a_{i})$. Alors\\ $(b_{j_{1}},b_{j_{2}},\cdots,b_{j_{n_{0}}})$ est une base de $H$ et $\mathcal{B}:=(b_{j_{1}},b_{j_{2}},\cdots,b_{j_{n_{0}}};a_{i_{1}},a_{i_{2}},\cdots,a_{i_{n_{1}}})$ est une base de $M$.\\
On a $\underline{n}_{0}:=\{j_{1},j_{2},\cdots,j_{n_{0}}\}$ et on suppose que $j_{1}<j_{2}<\cdots<j_{n_{0}}$. On pose $\underline{n}_{1}:=\{i_{1},i_{2},\cdots,i_{n_{1}}\}$ et $\underline{n}=\{1,2,\cdots,n\}$. Si $j\in \underline{n}_{0}$, on peut écrire:
\begin{equation}
b_{j}=a_{j}+\sum_{k\in \underline{n}_{1}}\rho_{j}^{k}a_{k}
\end{equation}
Le théorème suivant donne la valeur des $\rho_{j}^{k}$ en fonction des coefficients de Cartan $c_{i,j}$ et aussi les différentes relations existant entre eux.
\begin{theorem}\label{calcul des $c_{i,j}$}
Avec les hypothèses et notations précédentes, on a:\\
1)
\begin{equation}
  Car(G_{1})
  \begin{pmatrix}
\rho_{j_{1}}^{i_{1}} & \rho_{j_{1}}^{i_{2}} & \cdots & \rho_{j_{1}}^{i_{n_{1}}} \\
\rho_{j_{2}}^{i_{1}} & \rho_{j_{2}}^{i_{2}} & \cdots & \rho_{j_{2}}^{i_{n_{1}}} \\
\vdots & \vdots & & \vdots\\
\rho_{j_{n_{0}}}^{i_{1}} & \rho_{j_{n_{0}}}^{i_{2}} & \cdots & \rho_{j_{n_{0}}}^{i_{n_{1}}} 
\end{pmatrix}
= -
\begin{pmatrix}
c_{i_{1},j_{1}} & c_{i_{1},j_{2}} & \cdots & c_{i_{1},j_{n_{0}}}\\
c_{i_{2},j_{1}} & c_{i_{2},j_{2}} & \cdots & c_{i_{2},j_{n_{0}}}\\
\vdots & \vdots & & \vdots\\
c_{i_{n_{1}},j_{1}} & c_{i_{n_{1}},j_{2}} & \cdots & c_{i_{n_{1}},j_{n_{0}}}
\end{pmatrix}
 \end{equation}
 2) Une condition nécéssaire et suffisante pour que $H$ soit de dimension $n_{0}$ est:
 \begin{equation}
\begin{pmatrix}
c_{j_{1},i_{1}} & c_{j_{1},i_{2}} & \cdots & c_{j_{1},i_{n_{1}}}\\
c_{j_{2},i_{1}} & c_{j_{2},i_{2}} & \cdots & c_{j_{2},i_{n_{1}}}\\
\vdots & \vdots & & \vdots\\
c_{j_{n_{0}},i_{1}} & c_{j_{n_{o}},i_{2}} & \cdots & c_{j_{n_{0}},i_{n_{1}}}
\end{pmatrix}
Car(G_{1})^{-1}
\begin{pmatrix}
c_{i_{1},j_{1}} & c_{i_{1},j_{2}} & \cdots & c_{i_{1},j_{n_{0}}}\\
c_{i_{2},j_{1}} & c_{i_{2},j_{2}} & \cdots & c_{i_{2},j_{n_{0}}}\\
\vdots & \vdots & & \vdots\\
c_{i_{n_{1}},j_{1}} & c_{i_{n_{1}},j_{2}} & \cdots & c_{i_{n_{1}},j_{n_{0}}}
\end{pmatrix}
=Car(G_{0}).
\end{equation}
\end{theorem}
\begin{proof}
Soient $l\in \underline{n}$ et $j\in \underline{n_{0}}$. On a 
\[
s_{l}(b_{j})=b_{j}=a_{j}+\sum_{k\in \underline{n}_{1}}\rho_{j}^{k}a_{k}=(a_{j}-c_{l,j}a_{l})+\sum_{k\in \underline{n}_{1}}\rho_{j}^{k}(a_{k}-c_{l,k}a_{j}).
\]
Donc on a les relations:
\begin{equation}
\forall l \in \underline{n}, \forall j \in \underline{n_{0}}, c_{l,j}+\sum_{k\in \underline{n}_{1}}\rho_{j}^{k}c_{l,k}=0
\end{equation}
qui sont nécessaires et suffisantes pour que $H$ soit de dimension $n_{0}$. Les formules de l'énoncé sont simplement les transcriptions matricielles des relations (4).
\end{proof}
Avec les hypothèses et notations précédentes, on voit que $n_{1}$ est le rang de la matrice $Car(G)$. Si l'on prend n'importe quelle sous-matrice carrée $T$ de $Car(G)$ de dimension $n'$ avec $n_{1}<n'$, alors le rang de $T$ est encore $n_{1}$. Ces remarques démontrent la proposition suivante:
\begin{proposition}
Si l'on prend un sous-ensemble $S'$ de $S$ tel que $|S'|>n_{1}$ et si $G'=<S'>$ et $R'$ est la restriction de $R$ à $G'$ opérant sur $<a_{s}|s\in S'>$, alors $R'$ est une représentation de réflexion réductible de $G'$.
\end{proposition}
\begin{proof}
Car $Car(G')$ est une sous-matrice de $Car(G)$.
\end{proof}
Nous généralisons légèrement le problème. Nous travaillons dans la base $\mathcal{B}$. Soit $\mathcal{G}:=\{g|g\in GL(M),g\, \text{fixe} \,H\}$. La matrice de $g\in\mathcal{G}$ est:
\[
\begin{pmatrix}
I_{n_{0}} & A(g)\\
0 & P(g)
\end{pmatrix}
\]
où $I_{n_{0}}$ est la matrice identité d'ordre $n_{0}$, $A(g)\in \mathcal{M}(n_{0},n_{1})$ ($\mathcal{M}(n_{0},n_{1})$ étant l'espace vectoriel des matrices à $n_{0}$ lignes et $n_{1}$ colonnes à coefficients dans $K$) et $P(g)$ est un élément de $GL_{n_{1}}(K)$.\\
On pose 
\begin{align}
\mathcal{Z}:=\{z\in \mathcal{G}|z \,\text{opère comme une opération scalaire sur}\, M'\},\\
\mathcal{N}:=\{n\in \mathcal{Z}| n \, \text{opère comme l'identité sur }\, M'\},\\
\mathcal{Z'}:= \{z\in \mathcal{Z}|z \, \text{opère comme $-1$ sur}\, M'\}.
\end{align}
Si $g$ et $g'$ sont dans $\mathcal{G}$, alors 
\[
gg'=
\begin{pmatrix}
I_{n_{0}} & A(g')+A(g)P(g)\\
0 & P(g)p(g')
\end{pmatrix}
\]
donc $P: \mathcal{G}\to GL_{n_{1}}(K)$ est un morphisme surjectif de groupes dont le noyau est le groupe $\mathcal{N}$. On a aussi la formule $A(gg')=A(g')+A(g)P(g')$, de là on déduit que 
\[
g^{-1}=
\begin{pmatrix}
I_{n_{0}} & -A(g)P(g)^{-1}\\
0 & P(g)^{-1}.
\end{pmatrix}
\]
Soient $h$ et $h'$ dans $\mathcal{N}$. Alors 
\[
hh'=
\begin{pmatrix}
I_{n_{0}} & A(h)+A(h')\\
0 & I_{n_{1}}
\end{pmatrix}
\]
avec $A(h)\in \mathcal{M}(n_{0},n_{1})$ et $A: \mathcal{N}\to \mathcal{M}(n_{0},n_{1}):h \mapsto A(h)$ est un isomorphisme de groupes.\\
Notons le résultat suivant:
\begin{proposition}
Avec les hypothèses et notations du début, on a:\\
1) $\mathcal{G}$ opère à droite sur $\mathcal{M}(n_{0},n_{1})$ et sous l'action de $\mathcal{G}$ chaque sous-espace $\mathcal{M}_{i}$ de $\mathcal{M}(n_{0},n_{1})$ de la forme
\[
\begin{pmatrix}
 & & 0 & &\\
u_{1} & u_{2} & \cdots & u_{n_{1}-1} &u_{n_{1}}\\
 & & 0 & &
\end{pmatrix}
\]
où les éléments non nuls sont sur la i-ième  ligne, est stable par $\mathcal{G}$. On obtient ainsi une décomposition en somme directe de $K\mathcal{G}$-modules simples $\mathcal{M}_{i}$ $(1\leqslant i \leqslant n_{0})$ de $\mathcal{M}(n_{0},n_{1})$.\\
2) Si $\mathcal{N}_{i}=A^{-1}(\mathcal{M}_{i})$, alors chaque $\mathcal{N}_{i}$ est un $K\mathcal{G}$-module simple et $\mathcal{N}\simeq \oplus_{i=1}^{n_{0}}\mathcal{N}_{i}$.
\end{proposition}
\begin{proof}
Soient $g \in \mathcal{G}$ et $h \in \mathcal{N}$. Alors 
\[
ghg^{-1}=
\begin{pmatrix}
I_{n_{0}} & A(h)P(g)^{-1}\\
0 & I_{n_{1}}
\end{pmatrix}
\] 
donc $\mathcal{G}$ opère à droite sur $\mathcal{M}(n_{0},n_{1})$ et chaque $\mathcal{M}_{i}$ est stable sous l'action de $\mathcal{G}$. Tout le reste est clair.
\end{proof}

Avec les notations précédentes, la formule $A(ghg^{-1})=A(h)P(g)^{-1}$ montre que $A$ est $\mathcal{G}$-équivariante.

 Nous donnons maintenant une condition suffisante qui assure que $N(G)\neq\{1\}$.
 \begin{proposition}
Avec les hypothèses et notations précédentes, si $C_{G}(G_{1})\neq\{1\}$, en particulier si $Z(G_{1})\neq\{1\}$, on a $N(G)\neq\{1\}$.
\end{proposition}
\begin{proof}
Si $z\in C_{G}(G_{1})-\{1\}$, alors $z$ opère comme une opération scalaire sur $M'$, donc $z\in \mathcal{Z}\cap G$. Il en résulte que $\forall g\in G$, $[z,g]\in \mathcal{N}\cap G$ et comme $z\not \in Z(G)$ (qui est trivial), il existe $g\in G$ tel que $[z,g]\neq 1$, donc $N(G)\neq\{1\}$.
\end{proof}
\begin{corollary}
Avec les hypothèses et notations précédentes, on suppose en plus que $n=3$. Alors $N(G)\neq\{1\}$ si $2|pqr$.
\end{corollary}
\begin{proof}
Si $n=3$, on a $n_{0}=1$ et $n_{1}=2$ donc ou bien $G_{1}$ est isomorphe à $C_{2}\times C_{2}$ auquel cas le résultat est vrai ou bien $G_{1}$ est un groupe diédral et ceci quelque soit le choix de $S_{1}$. Comme $2|pqr$, l'un de ces groupes diédraux est d'ordre multiple de $4$, donc son centre est non trivial et nous appliquons la proposition 2 pour avoir le résultat.
\end{proof}
\begin{proposition}
Soient $W$ un groupe de Coxeter satisfaisant aux conditions $H(Cox)$ et $R :\to GL(M)$ une représentation de réflexion affine. On pose $G:=Im R$ et on suppose que $N(G)\neq{1}$. Soient $z \in \mathcal{Z}'$ et $\Gamma:=<G,z>$. Alors:\\
1) $N(\Gamma)$ = $N(G)$.\\
2) $|\Gamma /G|\leqslant 2$.
\end{proposition}
\begin{proof}
Il est clair que $N(G)\lhd \Gamma$ car $G$ normalise $N(G)$ et $z$ opère comme $-1$ sur $N(G)$.
Soit $\pi:\Gamma \to \Gamma/N(G)$ la projection canonique. Supposons que $N(\Gamma)\neq N(G)$ alors nous obtenons une contradiction. En effet $\pi(N(\Gamma))$ est un sous-groupe normal de $SL_{2}(M')$ puisque $\pi(N(\Gamma))$ est formé d'éléments de déterminants $1$. Il est clair que $N(\Gamma)\neq\Gamma$. les seuls sous-groupes normaux non triviaux de $SL_{2}(M')$ sont contenus dans son centre et ils sont formés d'applications scalaires. Comme tous les éléments de $\pi(N(\Gamma))$ n'ont que $1$ comme valeur propre, ils sont triviaux et $\pi(N(\Gamma))$ = $1$: on a $N(\Gamma)$ = $N(G)$. Comme $\forall g \in G$, $[z,g] \in N(\Gamma)=N(G)$, nous voyons que $z$ normalise $G$ d'où le résultat $|\Gamma /G|\leqslant 2$.
\end{proof}
\begin{corollary}
Avec les hypothèses et notations précédentes on a:\\
1) $\mathcal{N}$ et $\mathcal{Z'}$ normalisent $G$.\\
2) On a l'une des possibilités suivantes: $\mathcal{Z'} \subset G$ ou $\mathcal{Z'}\cap G = \emptyset$.
\end{corollary}
\begin{proof}
1) Comme $z$ était quelconque dans la proposition on voit que $\mathcal{Z'}$ normalise $G$ et comme $\mathcal{N}\subset <\mathcal{Z'}>$, $\mathcal{N}$ normalise $G$.\\
2) Si $\exists z \in G\bigcap \mathcal{Z'}$, alors $\forall z' \in \mathcal{Z'}$, $zz' \in N(G)$, donc $z' \in G$ et $\mathcal{Z'}\subset G$.
\end{proof}
\subsection{Opérations du groupe $G$}
On garde les notations précédentes. Il est clair que $G$ est un sous-groupe de $\mathcal{G}$ et que $N(G)$ est un sous-groupe de $\mathcal{N}$. Le groupe $G$ opère sur $N(G)$ (et sur $\mathcal{N}$) par conjugaison et sur $\mathcal{M}(n_{0},n_{1})$ par multiplication à droite. On note $g.$ l'opération de $g$ sur $\mathcal{M}(n_{0},n_{1})$.\\
L'application $P:G\to GL_{n_{1}}(K)$ est un morphisme de groupes dont le noyau est $N(G)$ et l'image est $G'$.\\
Nous donnons maintenant les matrices des éléments de $S$ dans la base $\mathcal{B}$ de $M$.\\
-- Si $i \in \underline{n}_{1}$, on a 
\[
s_{i}=
\begin{pmatrix}
I_{n_{0}} & 0 \\
0 & P(s_{i})
\end{pmatrix}
\]
Si $k \in \underline{n}_{1}$, on sait que $s_{i}(a_{k})=a_{k}-c_{ik}a_{i}$ donc $P(s_{i})=I_{n_{1}}+T(s_{i})$ où 
\[
T(s_{i})=
\begin{pmatrix}
& & 0 & &\\
-c_{ii_{1}} & -c_{ii_{2}} & \cdots & -c_{ii_{n_{1}-1}} & -c_{ii_{n_{1}}}\\
& & 0 & &
\end{pmatrix}
\]
et les éléments non nuls de $T(s_{i})$ sont situés sur la i-ième ligne.\\
Si $\zeta \in \mathcal{M}(n_{0},n_{1})$ on a $s_{i}.\zeta =\zeta P(s_{i})= \zeta+\zeta T(s_{i})$ d'où 
\[
[s_{i},\zeta]=s_{i}.\zeta-\zeta=\zeta T(s_{i}).
\]
-- Si $j \in \underline{n}_{0}$ et si $i \in \underline{n}_{1}$, on a $s_{j}(a_{i})=a_{i}-c_{ji}a_{j}$, mais $b_{j}=a_{j}+\sum_{k \in \underline{n}_{1}}\rho_{j}^{k}a_{k}$ et l'on obtient $s_{j}(a_{i})=a_{i}-c_{ji}b_{j}+c_{ji}\sum_{k \in \underline{n}_{1}}\rho_{j}^{k}a_{k}$. On en déduit que
\[
s_{j}=
\begin{pmatrix}
I_{n_{0}} & A(s_{j})\\
0 & P(s_{j})
\end{pmatrix}
\]
où $A(s_{j})\in \mathcal{M}(n_{0},n_{1})$ et toutes les lignes de $A(s_{j})$ sont formées de $0$ à l'exception de la j-ième qui est $(-c_{ji_{1}},-c_{ji_{2}},\cdots,-c_{ji_{n_{1}}})$.\\
En écrivant que $s_{j}^{2}=id_{M}$, on obtient 
\[
s_{j}^{2}=
\begin{pmatrix}
I_{n_{0}} & 0\\
0 & I_{n_{1}}
\end{pmatrix}
=
\begin{pmatrix}
I_{n_{0}} & A(s_{j})+A(s_{j})P(s_{j})\\
0 & I_{n_{1}}
\end{pmatrix}
\]
d'où la relation: (1)  $A(s_{j})+A(s_{j})P(s_{j})=0$. On a $P(s_{j})=I_{n_{1}}+T(s_{j})$ où 
\[
T(s_{j})=
\begin{pmatrix}
c_{ji_{1}}\rho_{j}^{i_{1}} & c_{ji_{2}}\rho_{j}^{i_{1}} & \cdots & c_{ji_{n_{1}}}\rho_{j}^{i_{1}}\\
c_{ji_{1}}\rho_{j}^{i_{2}} & c_{ji_{2}}\rho_{j}^{i_{2}} & \cdots & c_{ji_{n_{1}}}\rho_{j}^{i_{2}}\\
\vdots & \vdots & & \vdots\\
c_{ji_{1}}\rho_{j}^{i_{n_{1}}} & c_{ji_{2}}\rho_{j}^{i_{n_{1}}} & \cdots &c_{ji_{n_{1}}}\rho_{j}^{i_{n_{1}}}
\end{pmatrix}
.
\]
Si $\zeta \in \mathcal{M}(n_{0},n_{1})$, alors 
\[
s_{j}.\zeta=
\begin{pmatrix}
I_{n_{0}} & A(s_{j})+\zeta P(s_{j})+A(s_{j})P(s_{j})\\
0 & I_{n_{1}}
\end{pmatrix}
\]
d'où d'après la relation (1) $s_{j}.\zeta =\zeta P(s_{j})=\zeta+\zeta T(s_{j})$ et 
\[
[s_{j},\zeta]=\zeta T(s_{j}).
\]
\subsection{Le cas $n_{0}=1$}
Nous étudions en détail le cas $n_{0}=1$ puis nous passerons au cas $n_{0}>1$.\\
Comme $P(G)$ est un sous-groupe de $GL_{n_{1}}(K)$, $P(G)$ stabilise chacun des $\mathcal{M}_{j}$ $(j\in \underline{n}_{0})$ et pour voir comment $P(G)$ opère sur $\mathcal{M}_{j}$, nous allons d'abord supposer que $\underline{n}_{0}=\{1\}$ (donc $\underline{n}_{1}=\{2,3,\cdots,n\}$).\\
Par hypothèse $H$ est de dimension $1$ engendré par $b$ où, en posant pour simplifier les notations $\rho_{i}:=\rho_{1}^{i}\, (2\leqslant i \leqslant n)$,
\[
b=a_{1}+\sum_{i=2}^{n}\rho_{i}a_{i}
\]
et les $\rho_{i}$ sont donnés par (en utilisant la relation (2) du théorème  2):
\[
Car(G_{1})
\begin{pmatrix}
\rho_{2}\\
\vdots\\
\rho_{n}
\end{pmatrix}
=-
\begin{pmatrix}
c_{21}\\
\vdots\\
c_{n1}
\end{pmatrix}
\]
De plus on a les relations:
\[
\begin{pmatrix}
c_{12} & c_{13} & \cdots & c_{1n}
\end{pmatrix}
Car(G_{1})^{-1}
\begin{pmatrix}
c_{21}\\
c_{31}\\
\vdots\\
c_{n1}
\end{pmatrix}
=(2).
\]
-- Si $i\in \underline{n}_{1}$,
\[
T(s_{i})=
\begin{pmatrix}
& & 0 & &\\
-c_{i2} & -c_{i3} & \cdots & -c_{i(n-1)} & -c_{in}\\
& & 0 & &
\end{pmatrix}
\]
où les éléments non nuls sont sur la i-ième ligne.\\
Si $\zeta:=(\zeta_{2} \zeta_{3} \cdots \zeta_{n})\in \mathcal{M}(n_{0},n_{1})$ alors
\[
[s_{i},\zeta]=\zeta T(s_{i})=\zeta_{i}(-c_{i2}  -c_{i3}  \cdots  -c_{i(n-1)}  -c_{in}).
\]
-- \[
T(s_{1})=
\begin{pmatrix}
c_{12}\rho_{2} & c_{13}\rho_{2} & \cdots & c_{1n}\rho_{2}\\
c_{12}\rho_{3} & c_{13}\rho_{3} & \cdots & c_{1n}\rho_{3}\\
\vdots & \vdots & & \vdots\\
c_{12}\rho_{n} & c_{13}\rho_{n} & \cdots & c_{1n}\rho_{n}
\end{pmatrix}
.
\]
De plus:
\[
[s_{1},\zeta]=\omega(\zeta)(-c_{12}  -c_{13}  \cdots  -c_{1(n-1)}  -c_{1n})
\]
où $\omega(\zeta)=-\sum_{k=2}^{n}\zeta_{k}\rho_{k}$ et $\omega(c_{i})=-\sum_{k=2}^{n}c_{ik}\rho_{k}=c_{i1}\quad (i \in \underline{n}_{1})$.\\
Dans la suite , on pose $c_{i}:=(-c_{i2}  -c_{i3}  \cdots  -c_{i(n-1)}  -c_{in}) \, (1\leqslant i \leqslant n)$; on obtient ainsi les formules:
\[
[s_{i},c_{k}]=-c_{ki}c_{i} \quad (i \in \underline{n},k \in \underline{n}).
\]
\begin{theorem}
Avec les hypothèses et notations précédentes, on a:\\
1) $(c_{2},c_{3},\cdots,c_{n})$ est une base de $\mathcal{M}(n_{0},n_{1})(=\mathcal{M}(1,n-1))$.\\
2) On a $c_{1}=\sum_{i=2}^{n}\lambda_{i}c_{i}$ et les $\lambda_{i}$ sont donnés par
\[
\begin{pmatrix}
c_{12} & c_{13} & \cdots & c_{1n}
\end{pmatrix}
Car(G_{1})^{-1}=
\begin{pmatrix}
\lambda_{2} & \lambda_{3} & \cdots & \lambda_{n}
\end{pmatrix}
.
\]
3) $\mathcal{M}(n_{0},n_{1})$ est un $KG$-module simple isomorphe au $KG$-module $M'$. De plus $G$ opère comme un groupe de réflexion sur lui.
\end{theorem}
\begin{proof}
1) On a $\det (c_{2},c_{3}, \cdots,c_{n})=\det(Car(G_{1}))=\Delta(G_{1})\neq 0$ par hypothèse, donc $(c_{2},c_{3}, \cdots,c_{n})$ est un système libre. Comme $\mathcal{M}(n_{0},n_{1})$ est de dimension $n-1$, $(c_{2},c_{3}, \cdots,c_{n})$ est une base de $\mathcal{M}(n_{0},n_{1})$.\\
2) Si $c_{1}=\sum_{i=2}^{n}\lambda_{i}c_{i}$, alors les $\lambda_{i}$ sont donnés par les formules de l'énoncé.\\
3) Soient $M'=M/H$ et $\pi_{M}:M\to M'$ la projection canonique. On pose $a'_{i}:=\pi_{M}(a_{i})$ $(i\in \underline{n}_{1})$. Alors $(a'_{i})_{i\in\underline{n}_{1}}$ est une base de $M'$. On définit $\varphi:M'\to \mathcal{M}(n_{0},n_{1})$ par $\varphi(a'_{i}):=c_{i}$ $(i\in \underline{n}_{1})$.\\
Si $k\in \underline{n}$ et si $i\in \underline{n}_{1}$, on a $s_{k}.(\pi_{M}(a_{i}))=a'_{i}-c_{ki}a'_{k}=\pi_{M}(s_{k}(a_{i}))$;\\
 $s_{k}.\varphi(a'_{i})=s_{k}.c_{i}=c_{i}-c_{ki}c_{k}=\varphi(a'_{i}-c_{ki}a'_{k})=\varphi(s_{k}.a'_{i})$, donc $\varphi$ est $G$-équivariante.\\
 On a le diagramme de $KG$-modules:
 \[
 \begin{diagram}
\node{M} \arrow[1]{e,t}{\varphi \circ \pi_{M}} \arrow{s,r,l}{\pi_{M}}
\node{\mathcal{M}(n_{0},n_{1})}\\
\node{M'} \arrow{ne,b}{\varphi}
\end{diagram}
 \]
 Comme $M'$ est un $KG$-module simple, on voit que $\mathcal{M}(n_{0},n_{1})$ est aussi un $KG$-module simple isomorphe à $M'$. dans ces conditions, il est clair que $G$ opère comme un groupe de réflexions sur $\mathcal{M}(n_{0},n_{1})$.
\end{proof}
Nous étudions maintenant le cas général.
\begin{theorem}
On suppose que $n_{0}\geqslant 1$. Alors\\
1) $\forall j\in \underline{n}_{0}$, $M_{j}$ est un $KG$-module de réflexion simple isomorphe à $M'$ (en tant que $KG$-module).\\
2) Si $N(G)\neq \{1\}$ on a 
\begin{enumerate}
  \item $A(N(G))\cap M_{j} \neq \{1\}$ et $(A(N(G))\cap M_{j})\otimes_{\mathbb{Z}}K=M_{j}$;
  \item $A(N(G))\otimes_{\mathbb{Z}}K=\mathcal{M}(n_{0},n_{1})$.
  \end{enumerate}
\end{theorem}
\begin{proof}
1) Comme $P(G)$ est un sous-groupe de $GL_{n_{1}}(K)$, $P(G)$ stabilise chaque $M_{j}$ $(j\in \underline{n}_{0})$.\\
Soit $j\in \underline{n}_{0}$. On appelle $c_{i}(j)$ l'élément suivant de $M$, où $i\in \underline{n}_{1}$:
\[
c_{i}(j)=
\begin{pmatrix}
&& 0 &&\\
-c_{ii_{1}} & -c_{ii_{2}} & \cdots & -c_{ii_{(n_{1}-1)}} & -c_{ii_{n_{1}}}\\
&& 0 &&
\end{pmatrix}
.
\]
Soit \[\zeta :=
\begin{pmatrix}
&& 0 &&\\
\zeta_{i_{1}} & \zeta_{i_{2}} & \cdots &\zeta_{i_{(n_{1}-1)}} & \zeta_{i_{n_{1}}}\\
&& 0 &&
\end{pmatrix}\]
un élément de $M_{j}$. Si $i \in \underline{n}_{1}$, un calcul simple montre que $[s_{i},\zeta]=\zeta T(s_{i})=\zeta c_{i}(j)$. Nous avons les mêmes formules que dans le cas $n_{0}=1$, d'où le résultat.\\
2) Le groupe $G$ opère sur $N(G)$ par conjugaison et sur $\mathcal{M}(n_{0},n_{1})$ par translations à droite. De plus $A_{|N(G)}:N(G)\to A(N(G))$ est un isomorphisme de groupes, donc $A(N(G))$ se décompose en somme directe de sous-groupes sous l'action de $G$:
\[
A(N(G))=\oplus_{j\in \underline{n}_{0}}(A(N(G))\cap M_{j}).
\]
On pose $A(N(G))_{j}:=A(N(G))\cap M_{j}$ et $N(G)_{j}:=A^{-1}(A(N(G))_{j})$. Il en résulte que:
\[
N(G)=\oplus_{j\in \underline{n}_{0}}N(G)_{j}
\]
et chaque $N(G)_{j}$ est un sous-groupe normal de $G$ qui n'est pas central car $Z(G)=\{1\}$.\\
Chaque $M_{j}$ ($j\in \underline{n}_{0}$) est un espace vectoriel de dimension $n_{1}$ et chaque $s\in S$ opère linéairement sur lui; son polynôme caractéristique est $P_{s}(X)=(X+1)(X-1)^{n_{1}-1}$ donc chaque $s$ dans $S$ opère comme une réflexion sur $M_{j}$ et $[s,M_{j}]$ est de dimension $1$. Pour tout $s$ dans $S$ on donne un générateur de $[s,M_{j}]$. 
\begin{notation}
Soit $j\in \underline{n}_{0}$,
si $i\in \underline{n}$, on pose 
\[
c_{i}(j)=
\begin{pmatrix}
&& 0 &&\\
-c_{ii_{1}} & -c_{ii_{2}} & \cdots & -c_{ii_{n_{1}-1}} & -c_{ii_{n_{1}}}\\
&& 0 &&
\end{pmatrix}
\]
les éléments non nuls étant sur la j-ième ligne.
\end{notation}
 On montre maintenant que $\forall i\in \underline{n}$ on a $[s_{i},M_{j}]=<c_{i}(j)>$.\\
 - Soit $j\in \underline{n}_{0}$. Pour tout $h\in N(G)_{j}$, il existe $s$ dans $S$ tel que  $[s,h]\neq \{1\}$, donc $s[s,h]s^{-1}=[h,s]=[s,h]^{-1}$: $\forall s \in S,\exists h\in N(G)_{j}$ tel que $shs^{-1}=h^{-1}$ car $Z(G)=\{1\}$.\\
 Soit $s_{k}\in \underline{n}_{0}$. Comme on a la relation $A(s_{k})+A(s_{k})P(s_{k})=0$ on voit que $[s_{k},M_{j}]=<A(s_{k})>$ car $A(s_{k})\neq 0$ sinon $<G_{1},s_{k}>$ opérerait irréductiblement sur $M$ contrairement au choix de $G_{1}$. De plus, on voit que $A(s_{k})=c_{k}(j)$. On a le résultat dans ce cas.\\
 - Soit $i\in \underline{n}_{1}$. Alors $T(s_{i})=c_{i}(j)$ et si $h\in N(G)_{j}$ on a 
 \[
 A(h)=
 \begin{pmatrix}
&& 0 &&\\
h_{1} & h_{2} & \cdots & h_{n_{1}-1} & h_{n_{1}}\\
&& 0 &&
\end{pmatrix}
 \]
 où les éléments non nuls sont sur la j-ième ligne.\\
 On a $A(h)T(s_{i})=h_{i}c_{i}(j)$ et comme on a la relation $A(h)(I_{n_{1}}+P(s_{i}))=0=A(h)(2I_{n_{1}}+T(s_{i}))$, on obtient $2A(h)=-h_{i}c_{i}(j)$ et $A(h)=-\frac{h_{i}}{2}c_{i}(j)$. Il en résulte aussitôt que $[s_{i},M_{j}]=<c_{i}(j)>$.\\
  Le reste est clair car $A$ est un isomorphisme $G$-équivariant de $K$-espaces vectoriels.
\end{proof}
\subsection{Exemples}
Dans ce paragraphe nous donnons des exemples en dimension principalement  $4$. Le cas $n=3$ sera étudié en détail dans la partie suivant. On garde les hypothèses du paragraphe 1.
\subsubsection{Le cas $n=4$}
On suppose $n=4$ et $G=<s_{1},s_{2},s_{3},s_{4}>$. Nous énumérons tous les graphes connexes $\Gamma$ de cardinal $4$; chaque arête est décorée d'un entier $p\geqslant3$. On se donne un arbre couvrant  $\textit{T}$, $s_{1}$ une racine de cet arbre (qui sera un cercle plein)  et enfin des arêtes en gras que l'on ajoute à l'arbre pour obtenir le graphe $\Gamma$ . Pour le groupe de réflexion $G$ ainsi obtenu nous donnons la base adaptée, la matrice de Cartan ainsi que le discriminant de cette représentation. Enfin nous donnons les conditions nécessaires et suffisantes pour qu'il existe des applications sesquilinéaires non nulles $G$-invariantes avec $\sigma \in Aut K$ tel que $\sigma^{2}=id_{K}$ et le sous corps des points fixes de $\sigma$ contient $K_{0}$. Les applications obtenues sont $\sigma$-hermitiennes. Si $\sigma$ est l'identité de $K$ on obtient des applications bilinéaires symétriques. Pour chaque $p_{i}$ on choisit $\alpha_{i}$ une racine de $v_{p_{i}}(X)$.\\
	(I) Un chemin
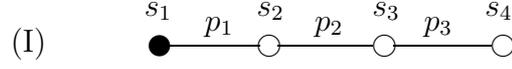
\[
\begin{picture}(150,88)
\put(-20,40){(I)}
\put(29,42){\circle*{7}}
\put(32,42){\line(1,0){30}}
\put(65,42){\circle{7}}
\put(68,42){\line(4,0){31}}
\put(103,42){\circle{7}}
\put(106,42){\line(7,0){32}}
\put(142,42){\circle{7}}
\put(24,52){$s_{1}$}
\put(61,52){$s_{2}$}
\put(99,52){$s_{3}$}
\put(136,52){$s_{4}$}
\put(44,47){$p_{1}$}
\put(80,47){$p_{2}$}
\put(116,47){$p_{3}$}
\end{picture}
\]
\[
Car(G)=
\begin{pmatrix}
2 & -\alpha_{1} & 0 & 0\\
-1 & 2 & -\alpha_{2} & 0\\
0 & -1 & 2 & \alpha_{3}\\
0 & 0 & -1 & 2
\end{pmatrix}
,
\]
\[
 (\varphi(a_{i},a_{j})_{1\leqslant i,j\leqslant4})=
\begin{pmatrix}
2 & -\alpha_{1} & 0 & 0\\
-\alpha_{1} & 2\alpha_{1} & -\alpha_{1}\alpha_{2} & 0\\
0 & -\alpha_{1}\alpha_{2} & 2\alpha_{1}\alpha_{2} & -\alpha_{1}\alpha_{2}\alpha_{3}\\
0 & 0 & -\alpha_{1}\alpha_{2}\alpha_{3} & 2\alpha_{1}\alpha_{2}\alpha_{3}
\end{pmatrix}
\]
On a $\Delta(G)=16-4\alpha_{1}-4\alpha_{2}-4\alpha_{3}+\alpha_{1}\alpha_{3}=(4-\alpha_{1})(4-\alpha_{3})-4\alpha_{2}$. \\
	(II) Une étoile
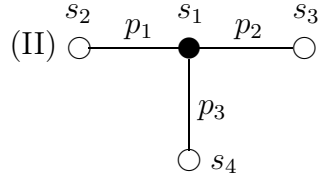
\[
\begin{picture}(150,88)
\put(6,40){(II)}
\put(29,42){\circle{7}}
\put(32,42){\line(1,0){30}}
\put(65,42){\circle*{7}}
\put(68,42){\line(4,0){31}}
\put(103,42){\circle{7}}
\put(65,8){\line(0,1){30}}
\put(65,5){\circle{7}}
\put(24,52){$s_{2}$}
\put(61,52){$s_{1}$}
\put(99,52){$s_{3}$}
\put(72,2){$s_{4}$}
\put(44,47){$p_{1}$}
\put(80,47){$p_{2}$}
\put(68,20){$p_{3}$}
\end{picture}
\]
\[
Car(G)=
\begin{pmatrix}
2 & -\alpha_{1} & -\alpha_{2} & -\alpha_{3}\\
-1 & 2 & 0 & 0\\
-1 & 0 & 2 & 0\\
-1 & 0 & 0 & 2
\end{pmatrix}
\]
\[
 (\varphi(a_{i},a_{j})_{1\leqslant i,j\leqslant4})=
 \begin{pmatrix}
 2 & -\alpha_{1} & -\alpha_{2} & -\alpha_{3}\\
 -\alpha_{1} & 2\alpha_{1} & 0 & 0\\
 -\alpha_{2} & 0 & 2\alpha_{2} & 0\\
 -\alpha_{3} & 0 & 0 & 2\alpha_{3}
\end{pmatrix}
\]
On a $\Delta(G) = 4(4-\sum_{1}^{4}\alpha_{i})$.\\
	(III) Un seul circuit qui n'est pas un carré
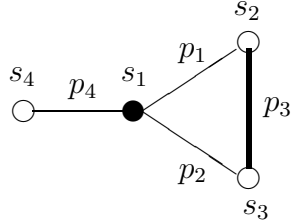
\[
\begin{picture}(150,88)
\put(29,42){\circle{7}}
\put(32,42){\line(1,0){30}}
\put(65,42){\circle*{7}}
\put(68,42){\line(3,2){31}}
\put(103,65){\circle{7}}
\linethickness{0.6mm}
\put(68,42){\line(3,-2){31}}
\put(24,52){$s_{4}$}
\put(61,52){$s_{1}$}
\put(98,74){$s_{2}$}
\put(101,8){$s_{3}$}
\put(103,20){\circle{7}}
\put(103,23){\line(0,1){40}}
\put(44,48){$p_{4}$}
\put(80,60){$p_{1}$}
\put(80,20){$p_{2}$}
\put(108,42){$p_{3}$}
\end{picture}
\]
\[
Car(G)=
\begin{pmatrix}
2 & -\alpha_{1} & -\alpha_{2} & -\alpha_{4}\\
-1 & 2 & -l & 0\\
-1 & -m & 2 & 0\\
-1 & 0 & 0 & 2
\end{pmatrix}
\]
avec $lm=\alpha_{3}$.\\
Si $\varphi$ $\neq 0$ $\in \Phi$, on a $\sigma(l)\alpha_{1}=m\alpha_{2}$ et
\[
(\varphi(a_{i},a_{j})_{1\leqslant i,j\leqslant4})=
\begin{pmatrix}
2 & -\alpha_{1} & -\alpha_{2} & -\alpha_{4}\\
-\alpha_{1} & 2\alpha_{1} & -m\alpha_{2} & 0\\
-\alpha_{2} & -l\alpha_{1} & 2\alpha_{2} & 0\\
-\alpha_{4} & 0 & 0 & 2\alpha_{4}
\end{pmatrix}
\]
On a $\Delta(G) = 4(4-\sum_{1}^{4}\alpha_{i})+\alpha_{3}\alpha_{4}-2(\alpha_{1}l+\alpha_{2}m)$;\\
si $\varphi$ est bilinéaire alors $\Delta(G) = 4(4-\sum_{1}^{4}\alpha_{i})+\alpha_{3}\alpha_{4}-4\alpha_{1}l$\\
	(IV) Un carré
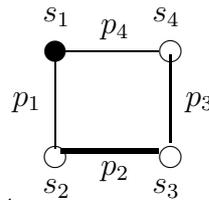
\[
\begin{picture}(150,88)
\put(65,7){\circle{7}}
\put(65,10){\line(0,1){30}}
\put(65,42){\circle*{7}}
\put(68,42){\line(4,0){31}}
\put(103,42){\circle{7}}
\linethickness{0.6mm}
\put(67,9){\line(1,0){32}}
\put(103,7){\circle{7}}
\thinlines
\put(103,12){\line(0,1){29}}
\put(61,52){$s_{1}$}
\put(97,52){$s_{4}$}
\put(61,-5){$s_{2}$}
\put(97,-5){$s_{3}$}
\put(80,48){$p_{4}$}
\put(80,1){$p_{2}$}
\put(51,24){$p_{1}$}
\put(108,24){$p_{3}$}
\end{picture}
\]
\[
Car(G)=
\begin{pmatrix}
2 & -\alpha_{1} & 0 & -\alpha_{4}\\
-1 & 2 & -l & 0\\
0 & -m & 2 & -1\\
-1 & 0 & -\alpha_{3} & 2
\end{pmatrix}
\]
avec $lm=\alpha_{2}$.\\
Si $\varphi$ $\neq 0$ $\in \Phi$, on a $\sigma(l)\alpha_{1}=m\alpha_{3}\alpha_{4}$ et 
\[
(\varphi(a_{i},a_{j})_{1\leqslant i,j\leqslant4})=
\begin{pmatrix}
2 & -\alpha_{1} & 0 & -\alpha_{4}\\
-\alpha_{1} & 2\alpha_{1} & -m\alpha_{3}\alpha_{4} & 0\\
0 & -l\alpha_{1} & 2\alpha_{3}\alpha_{4} & -\alpha_{3}\alpha_{4}\\
-\alpha_{4} & 0 & -\alpha_{3}\alpha_{4} & 2\alpha_{4}
\end{pmatrix}
\]
$\Delta(G)=4(4-\sum_{1}^{4}\alpha_{i})+(\alpha_{1}\alpha_{3}+\alpha_{2}\alpha_{4})-(\alpha_{1}l+\alpha_{3}\alpha_{4}m)$;\\
si $\varphi$ est bilinéaire alors $\Delta(G) = 4(4-\sum_{1}^{4}\alpha_{i})+(\alpha_{1}\alpha_{3}+\alpha_{2}\alpha_{4})-2\alpha_{1}l$.\\
	(V) Un graphe avec deux circuits:
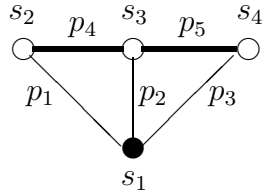
\[
\begin{picture}(150,88)
\put(29,42){\circle{7}}
\linethickness{0.6mm}
\put(32,42){\line(1,0){30}}
\put(65,42){\circle{7}}
\put(68,42){\line(4,0){31}}
\put(103,42){\circle{7}}
\put(65,9){\circle*{7}}
\thinlines
\put(65,11){\line(0,1){28}}
\put(31,40){\line(1,-1){30}}
\put(68,10){\line(1,1){30}}
\put(24,52){$s_{2}$}
\put(61,52){$s_{3}$}
\put(99,52){$s_{4}$}
\put(61,-3){$s_{1}$}
\put(44,48){$p_{4}$}
\put(80,48){$p_{5}$}
\put(30,25){$p_{1}$}
\put(90,25){$p_{3}$}
\put(67,25){$p_{2}$}
\end{picture}
\]
\[
Car(G)=
\begin{pmatrix}
2 & -\alpha_{1} & -\alpha_{2} & -\alpha_{3}\\
-1 & 2 & -l_{1} & 0\\
-1 & -m_{1} & 2 & -l_{2}\\
-1 & 0 & -m_{2} & 2
\end{pmatrix}
\]
avec $l_{1}m_{1}=\alpha_{4}$ et $l_{2}m_{2}=\alpha_{5}$.\\
Si $\varphi$ $\neq 0$ $\in \Phi$, on a $\sigma(l_{1})\alpha_{1}=\alpha_{2}m_{1}$ et $\sigma(l_{2})\alpha_{2}=\alpha_{3}m_{2}$ et 
\[
(\varphi(a_{i},a_{j})_{1\leqslant i,j\leqslant4})=
\begin{pmatrix}
2 & -\alpha_{1} & -\alpha_{2} & \alpha_{3}\\
-\alpha_{1} & 2\alpha_{1} & -m_{1}\alpha_{2} & 0\\
-\alpha_{2} & -\alpha_{1}l_{1} & 2\alpha_{2} & -\alpha_{3}m_{2}\\
-\alpha_{3} & 0 & -\alpha_{2}l_{2} & 2\alpha_{3}
\end{pmatrix}
\]
\begin{multline*}
\Delta(G)=4(4-\sum_{i=1}^{5}\alpha_{i})+(\alpha_{1}\alpha_{5}+\alpha_{3}\alpha_{4})\\-(2\alpha_{1}l_{1}+2\alpha_{2}l_{2}+2\alpha_{2}m_{1}+2\alpha_{3}m_{2})-(\alpha_{1}l_{1}l_{2}+\alpha_{3}m_{1}m_{2})
\end{multline*}
si $\varphi$ est bilinéaire alors $\Delta(G) = 4(4-\sum_{1}^{4}\alpha_{i})+(\alpha_{1}\alpha_{5}+\alpha_{3}\alpha_{4})-4(\alpha_{1}l_{1}+\alpha_{2}l_{2})-2\alpha_{1}l_{1}l_{2}$.\\
\pagebreak
	(VI) Le graphe complet.
\[
\begin{picture}(150,-20)
\put(70,7){$s_{1}$}
\put(31,40){\line(1,-1){30}}
\put(24,52){$s_{2}$}
\put(66,-30){\line(0,1){34}}
\put(29,42){\circle{7}}
\linethickness{0.6mm}
\put(30,40){\line(1,-2){35}}
\put(67,-30){\line(1,2){35}}
\put(32,42){\line(1,0){60}}
\put(68,42){\line(4,0){31}}
\put(103,42){\circle{7}}
\put(65,9){\circle*{7}}
\put(66,-33){\circle{7}}
\put(64,-45){$s_{4}$}
\put(99,52){$s_{3}$}
\put(31,40){\line(1,-1){30}}
\put(68,10){\line(1,1){30}}
\put(62,48){$p_{4}$}
\put(90,0){$p_{5}$}
\put(34,0){$p_{6}$}
\put(50,25){$p_{1}$}
\put(70,25){$p_{2}$}
\put(69,-5){$p_{3}$}
\end{picture}
\]
\linebreak[4]
\linebreak[4]
\[
Car(G)
\begin{pmatrix}
2 & -\alpha_{1} & -\alpha_{2} & -\alpha_{3}\\
-1 & 2 & -l_{4} & -m_{6}\\
-1 & -m_{4} & 2 & -l_{5}\\
-1 & -l_{6} & -m_{5} & 2
\end{pmatrix}
\]
avec $l_{4}m_{4}=\alpha_{4}$, $l_{5}m_{5}=\alpha_{5}$ et $l_{6}m_{6}=\alpha_{6}$.\\
Si $\varphi$ $\neq 0$ $\in \Phi$, on a $\sigma(l_{4})\alpha_{1}=m_{4}\alpha_{2}$, $\sigma(l_{5})\alpha_{2}=m_{5}\alpha_{3}$ et $\sigma(l_{6})\alpha_{3}=m_{6}\alpha_{1}$ et
\[(\varphi(a_{i},a_{j})_{1\leqslant i,j\leqslant4})=
\begin{pmatrix}
2 & -\alpha_{1} & -\alpha_{2} & \alpha_{3}\\
-\alpha_{1} & 2\alpha_{1} & -m_{4}\alpha_{2} & -l_{6}\alpha_{3}\\
-\alpha_{2} & -\alpha_{1}l_{4} & 2\alpha_{2} & -\alpha_{3}m_{5}\\
-\alpha_{3} & -m_{6}\alpha_{1} & -\alpha_{2}l_{5} & 2\alpha_{3}
\end{pmatrix}
\]
\begin{multline*}
\Delta(G)=4(4-\sum_{i=1}^{5}\alpha_{i})+(\alpha_{1}\alpha_{5}+\alpha_{3}\alpha_{4}+\alpha_{2}\alpha_{6})-2(l_{4}l_{5}l_{6}+m_{4}m_{5}m_{6})\\
-\alpha_{1}(2l_{4}+2m_{6}+l_{4}l_{5}+m_{5}m_{6})-\alpha_{2}(2l_{5}+2m_{4}+l_{5}l_{6}+m_{6}m_{4})-\alpha_{3}(2l_{6}+2m_{5}+l_{6}l_{4}+m_{4}m_{5})
\end{multline*}
Si $\varphi \neq 0$ est bilinéaire, $\Delta(G)$ devient:
\begin{multline*}
\Delta(G)=4(4-\sum_{i=1}^{5}\alpha_{i})+(\alpha_{1}\alpha_{5}+\alpha_{3}\alpha_{4}+\alpha_{2}\alpha_{6})-4l_{4}l_{5}l_{6}\\
-4(\alpha_{1}l_{4}+\alpha_{2}l_{5}+\alpha_{3}l_{6})-2(\alpha_{1}l_{4}l_{5}+\alpha_{2}l_{5}l_{6}+\alpha_{3}l_{6}l_{4}).
\end{multline*}
\textbf{Dans les propositions qui suivent on suppose toujours que $\Delta(G)=0$.}
\begin{proposition}
Avec les hypothèses et notations précédentes.\\
1) Dans le cas (I) on a une seule possibilité: $n_{0}=1$ et $p_{1}=p_{3}=4$, $p_{2}=3$, $\alpha_{1}=\alpha_{3}=2$, $\alpha_{2}=1$. On a le graphe $\Gamma(G)\simeq \Gamma(\tilde{C}_{3})$ et $G\simeq W(\tilde{C}_{3})$ groupe de Weyl affine de type $\tilde{C}_{3}$. De plus $b=a_{1}+a_{2}+a_{3}+\frac{1}{2}a_{4}$.\\
2) Dans le cas (II) on a deux possibilités, $n_{0}=1$ et\\
	(i) $p_{1}=p_{2}=3$, $p_{3}=4$, $\alpha_{1}=\alpha_{2}=1$, $\alpha_{3}=2$. On a le graphe $\Gamma(G)\simeq \Gamma(\tilde{B}_{3})$ et $G\simeq W(\tilde{B}_{3})$ groupe de Weyl affine de type $\tilde{B}_{3}$. De plus $b=a_{1}+\frac{1}{2}(a_{2}+a_{3}+a_{4})$.\\
	(ii) $p_{1}=p_{2}=5$, $p_{3}=3$, $\alpha_{1}=\tau$, $\alpha_{2}=3-\tau$, $\alpha_{3}=1$. On a $G\simeq W(\tilde{H}_{3})$ groupe de réflexion affine de type $H_{3}$. De plus $b=a_{1}+\frac{1}{2}(a_{2}+a_{3}+a_{4})$. Ce groupe sera étudié plus loin.
\end{proposition}
\begin{proof}
(I) On a $\Delta(G)=16-4\alpha_{1}-4\alpha_{2}-4\alpha_{3}+\alpha_{1}\alpha_{3}=(4-\alpha_{1})(4-\alpha_{3})-4\alpha_{2}=0$. D'après la proposition 13 de la partie 2, ceci n'est possible que si $\alpha_{1}=\alpha_{3}=2$ et $\alpha_{2}=1$, c'est à dire $p_{1}=p_{3}=4$ et $p_{2}=3$. On a donc un unique exemple :$G\simeq W(\tilde{C}_{3})$. En particulier on ne peut pas avoir $n_{0}=2$.\\
(II) On a $\Delta(G)=4(4-\alpha_{1}-\alpha_{2}-\alpha_{3})=0$. Toujours d'après la proposition 13 de la partie 2, ceci n'est possible que dans les deux cas suivants:\\
(i) $\alpha_{1}=\alpha_{2}=1$, $\alpha_{3}=2$ (par exemple) d'où $p_{1}=p_{2}=3$ et $p_{3}=4$. On a donc $G\simeq W(\tilde{B}_{3})$.\\
(ii) $\alpha_{1}=\tau,\alpha_{2}=3-\tau$ et $\alpha_{3}=1$, d'où $p_{1}=p_{2}=5$ et $p_{3}=3$. On a le résultat de l'énoncé. \\
Dans les deux cas, on ne peut pas avoir  $n_{0}=2$.
\end{proof}
\begin{proposition}
Avec les hypothèses et notations précédentes, on a une seule possibilité dans le cas (III): $n_{0}=1$. Si $G_{1}=<s_{1},s_{2},s_{3}>$, $\Delta(G_{1})=\frac{1}{2}\alpha_{4}(4-\alpha_{3})=8-2\alpha_{1}-2\alpha_{2}-2\alpha_{3}-(\alpha_{1}l+\alpha_{2}m)\neq 0$. $\alpha_{1}l$ et $\alpha_{2}m$ sont racines d'un polynôme $Q(X)\in K_{0}[X]$ du second degré.
On a: $b=a_{1}+\frac{l+2}{4-\alpha_{3}}a_{2}+\frac{m+2}{4-\alpha_{3}}a_{3}+\frac{1}{2(4-\alpha_{3})}a_{4}$.
\end{proposition}
\begin{proof}
On a $\Delta(G)=-\alpha_{4}(4-\alpha_{3})+2(8-2\alpha_{1}-2\alpha_{2}-2\alpha_{3}-(\alpha_{1}l+\alpha_{2}m))=0$. Comme $\alpha_{4}(4-\alpha_{3})\neq 0$ par les hypothèses générales, si $G_{1}=<s_{1},s_{2},s_{3}>$, on obtient $\Delta(G_{1})\neq 0$ donc $n_{0}=1$. On  a $\Delta(G_{1})=\frac{1}{2}\alpha_{4}(4-\alpha_{3})=8-2\alpha_{1}-2\alpha_{2}-2\alpha_{3}-(\alpha_{1}l+\alpha_{2}m)$ et $\alpha_{1}l$ et $\alpha_{2}m$ sont racines d'un polynôme $Q(X)\in K_{0}[X]$ du second degré. On a par un calcul facile:
\[
b=a_{1}+\frac{l+2}{4-\alpha_{3}}a_{2}+\frac{m+2}{4-\alpha_{3}}a_{3}+\frac{1}{2(4-\alpha_{3})}a_{4}.
\]
\end{proof}
Dans les cas (IV) et (V) soient $i\in \underline{n}$, $S_{i}=S-\{i\}$ et $R_{i}$ la restriction de $R$ à $R_{i}$ opérant sur $<\mathcal{A}-\{a_{i}\}>$. Si $n_{0}=1$ on ne peut pas avoir $R_{i}$ réductible pour tout $i$, donc il existe $i\in \underline{n}$ tel que $R_{i}$ est irréductible.\\
On s'intéresse d'abord au cas (IV). 
\[
\Delta(G)= 16-4\alpha_{1}-4\alpha_{2}-4\alpha_{3}-4\alpha_{4}+\alpha_{1}\alpha_{3}+\alpha_{2}\alpha_{4}-\alpha_{1}l-\alpha_{3}\alpha_{4}m=0.
\]
\begin{proposition}
Avec les hypothèses et notations précédentes, on a deux possibilités dans le cas (IV).\\
1) On suppose que $n_{0}=1$ et que $\underline{n}_{0}=\{1\}$. On a $G_{1}=<s_{2},s_{3},s_{4}>$ (on peut faire ce choix d'après la remarque précédente).\\
La formule donnant $\Delta(G)$ montre que $\alpha_{1}l$ et $\alpha_{3}\alpha_{4}m$ sont racines du polynôme $Q(X)\in K_{0}[X]$:
\[
Q(X)=X^{2}-(16-4\alpha_{1}-4\alpha_{2}-4\alpha_{3}-4\alpha_{4}+\alpha_{1}\alpha_{3}+\alpha_{2}\alpha_{4})X+\alpha_{1}\alpha_{2}\alpha_{3}\alpha_{4}.
\]
On voit ainsi que $[K:K_{0}]\leqslant 2$ et aussi que 
\[
b=a_{1}+\frac{1}{\Delta(G_{1})}[(4-\alpha_{3}+l)a_{2}+(2m+2)a_{3}+(m\alpha_{3}+4-\alpha_{2})a_{4}].
\]
2) On suppose maintenant que $n_{0}=2$: $\underline{n}_{0}=\{1,3\}$, $\underline{n}_{1}=\{2,4\}$. On a:
\[
G_{0}=<s_{1},s_{3}>,\,Car(G_{0})=
\begin{pmatrix}
2 & 0\\
0 & 2
\end{pmatrix}
\]
\[
G_{1}=<s_{2},s_{4}>,\,Car(G_{1})=
\begin{pmatrix}
2 & 0\\
0 & 2
\end{pmatrix}
\]
On a $b_{1}=a_{1}+\frac{1}{2}(a_{2}+a_{4}),b_{3}=a_{3}+\frac{1}{2}(la_{2}+\alpha_{3})$.\\
On a $p_{2}=p'=p_{4}$ et $p_{1}=p_{3}=p$; $\alpha_{2}=\alpha_{4}=4-\alpha_{1}$ et $\alpha_{3}=\alpha_{1}$. De plus $m=-1$ et $l=-\alpha_{2}$.
\end{proposition}\begin{proof}
1) Il suffit de faire les calculs.
2) On utilise les résultats généraux pour trouver la valeur de $b_{1}$ et de $b_{3}$. Enfin on a la relation:
\[
\begin{pmatrix}
-\alpha_{1} & -\alpha_{4}\\
-m & -1
\end{pmatrix}
\frac{1}{2}
\begin{pmatrix}
1 & 0\\
0 & 1
\end{pmatrix}
\begin{pmatrix}
-1 & -l\\
-1 & -\alpha_{3}
\end{pmatrix}
=\begin{pmatrix}
2 & 0\\
0 & 2
\end{pmatrix}
\]
d'où $\alpha_{1}+\alpha_{4}=4$, $m+1=0$, $l\alpha_{1}+\alpha_{3}\alpha_{4}=0$, $\alpha_{2}+\alpha_{3}=4$. On en déduit: $4-\alpha_{1}-\alpha_{4}=0$, $4-\alpha_{2}-\alpha_{3}=0$, $m=-1$, $l=-\alpha_{2}$ d'où $\alpha_{1}\alpha_{2}=\alpha_{3}\alpha_{4}=\alpha_{3}(4-\alpha_{1})=(4-\alpha_{2})(4-\alpha_{1})$ donc $16-4\alpha_{1}-4\alpha_{2}=0$ et $4-\alpha_{1}-\alpha_{2}=0$; on voit aussi que $4-\alpha_{3}-\alpha_{4}=0$. On a donc tous les résultats.
\end{proof}
On peut voir que le sous-groupe $<s_{1},s_{2},s_{3}>$ est un groupe diédral affine (voir la partie 3) et la structure de $G$ est maintenant immédiate.\\
Dans le cas (V), on a
\begin{multline*}
\Delta(G)=4(4-\sum_{i=1}^{5}\alpha_{i})+(\alpha_{1}\alpha_{5}+\alpha_{3}\alpha_{4})\\-(2\alpha_{1}l_{1}+2\alpha_{2}l_{2}+2\alpha_{2}m_{1}+2\alpha_{3}m_{2})-(\alpha_{1}l_{1}l_{2}+\alpha_{3}m_{1}m_{2})
\end{multline*}
-- Si $n_{0}=1$, alors $\Delta(G)=0$ et si l'on se donne $l_{1}$ quelconque (mais différent de $0$) alors $m_{1}$ est déterminé et $l_{2}$ et $m_{2}$ sont racines d'un polynôme $Q(X)$ du second degré à coefficients dans $K_{0}(l_{1})$. \textbf{Il y a ainsi une infinité de représentations de réflexion réductibles non équivalentes}.\\
-- On considère maintenant le cas $n_{0}=2$.
\begin{proposition}
On choisit $\underline{n}_{0}=\{1,3\}$ et $\underline{n}_{1}=\{2,4\}$. On a $b_{1}=a_{1}+\frac{1}{2}(a_{2}+a_{4})$ et $b_{3}=a_{3}+\frac{1}{2}(l_{1}a_{2}+m_{2}a_{4})$. On a $p_{3}=p_{1}'$, $p_{5}=p'_{4}$, $\alpha_{1}+\alpha_{3}=4$, $\alpha_{4}+\alpha_{5}=4$, $\alpha_{1}l_{1}+\alpha_{3}m_{2}+2\alpha_{2}=0$ et enfin $m_{1}+l_{2}+2=0$. Le $R$-sous-groupe $<s_{1},s_{2},s_{3}>$ (resp. $<s_{1},s_{3},s_{4}>$) est un  groupe de réflexion affine quelconque et on l'a plongé dans un groupe de réflexion affine de rang $4$ tel que $H$ soit de dimension $2$.
\end{proposition}
\begin{proof}
On a 
\[
Car(G_{1})=2
\begin{pmatrix}
1 & 0\\
0 & 1
\end{pmatrix}
, Car(G_{0})=
\begin{pmatrix}
2 & -\alpha_{2}\\
-1 & 2
\end{pmatrix}
\]
et nous obtenons:
\[
\begin{pmatrix}
\rho_{1}^{2} & \rho_{3}^{2}\\
\rho_{1}^{4} & \rho_{3}^{4}
\end{pmatrix}
=\frac{1}{2}
\begin{pmatrix}
1 & l_{1}\\
1 & m_{2}
\end{pmatrix}
\]
d'où les valeurs de $b_{1}$ et de $b_{3}$. De plus on a les relations de l'énoncé. Comme tous les mineurs de rang $3$ de $Car(G)$ sont nuls, on a aussi: $\Delta(<s_{1},s_{2},s_{3}>)=0$ et $\Delta(<s_{1},s_{3},s_{4}>)=0$ et ceci démontre la proposition.
\end{proof}
\subsubsection{Le cas $n=5$}
On ne s'intéresse ici uniquement au cas où $\Gamma(G)$ est un chemin:\\ $G=<s_{i}|1\leqslant i\leqslant 5>$ et 
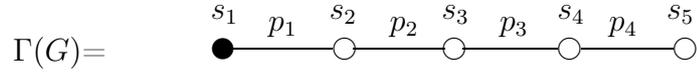
\[
\begin{picture}(150,88)
\put(-80,38){$\Gamma(G)$=}
\put(-11,42){\circle*{7}}
\put(-8,42){\line(1,0){33}}
\put(29,42){\circle{7}}
\put(32,42){\line(1,0){30}}
\put(65,42){\circle{7}}
\put(68,42){\line(1,0){31}}
\put(103,42){\circle{7}}
\put(107,42){\line(1,0){30}}
\put(141,42){\circle{7}}
\put(-15,52){$s_{1}$}
\put(24,52){$s_{2}$}
\put(61,52){$s_{3}$}
\put(99,52){$s_{4}$}
\put(135,52){$s_{5}$}
\put(4,48){$p_{1}$}
\put(44,48){$p_{2}$}
\put(80,48){$p_{3}$}
\put(116,48){$p_{4}$}
\end{picture}
\]
Pour $i\in \{1,2,3,4\}$ soit $\alpha_{i}$ une racine de $v_{p_{i}}(X)$. On a
\[
Car(G)=
\begin{pmatrix}
2 & -\alpha_{1} & 0 & 0 & 0\\
-1 & 2 & -\alpha_{2} & 0 & 0\\
0 & -1 & 2 & -\alpha_{3} & 0\\
0 & 0 & -1 & 2 & -\alpha_{4}\\
0 & 0 & 0 & -1 & 2
\end{pmatrix}
\]
et \[
\Delta(G)=2[(4-\alpha_{1})(4-\alpha_{4})-\alpha_{3}(4-\alpha_{1})-\alpha_{2}(4-\alpha_{4})].
\]
Je connais trois solutions à l'équation $\Delta(G)=0$ et je ne sais pas s'il y en a d'autres.\\
1) $p_{1}=p_{4}=4$, $p_{2}=p_{3}=3$ alors $\alpha_{1}=\alpha_{4}=2$ et $\alpha_{2}=\alpha_{3}=1$. 
\[
\begin{picture}(150,88)
\put(-80,38){$\Gamma(G)$=}
\put(-11,42){\circle{7}}
\put(-8,42){\line(1,0){33}}
\put(29,42){\circle{7}}
\put(32,42){\line(1,0){30}}
\put(65,42){\circle{7}}
\put(68,42){\line(1,0){31}}
\put(103,42){\circle{7}}
\put(107,42){\line(1,0){30}}
\put(141,42){\circle{7}}
\put(-15,52){$s_{1}$}
\put(24,52){$s_{2}$}
\put(61,52){$s_{3}$}
\put(99,52){$s_{4}$}
\put(135,52){$s_{5}$}
\put(4,48){$4$}
\put(44,48){$3$}
\put(80,48){$3$}
\put(116,48){$4$}
\end{picture}
\]
On obtient le groupe de Weyl affine de type $\tilde{C}_{4}$:
$G\simeq W(\tilde{C}_{4})$.\\
2) $p_{1}=p_{4}=3$, $p_{2}=p_{3}=5$ alors $\alpha_{1}=\alpha_{4}=1$ et $\alpha_{2}=\tau$ et $\alpha_{3}=3-\tau$. 
\[
\begin{picture}(150,88)
\put(-80,38){$\Gamma(G)$=}
\put(-11,42){\circle{7}}
\put(-8,42){\line(1,0){33}}
\put(29,42){\circle{7}}
\put(32,42){\line(1,0){30}}
\put(65,42){\circle{7}}
\put(68,42){\line(1,0){31}}
\put(103,42){\circle{7}}
\put(107,42){\line(1,0){30}}
\put(141,42){\circle{7}}
\put(-15,52){$s_{1}$}
\put(24,52){$s_{2}$}
\put(61,52){$s_{3}$}
\put(99,52){$s_{4}$}
\put(135,52){$s_{5}$}
\put(4,48){$3$}
\put(44,48){$5$}
\put(80,48){$5$}
\put(116,48){$3$}
\end{picture}
\]
Un calcul simple montre que $(s_{2}s_{4}^{s_{3}})^{3}=1$, donc d'après la proposition 16 2) de la section 4, $G_{1}=<s_{1},s_{2},s_{3},s_{4}>\simeq W(H_{4})$ et il existe $z\in Z(G_{1})-\{1\}$, d'où d'après la proposition 3, $(zs_{5})^{2}\in N(G)$. On a $G\simeq N(G)\ltimes W(H_{4})$ et $N(G)$ est un groupe commutatif libre de rang $8$.\\
3) $p_{1}=p_{4}=5$, $p_{2}=p_{3}=3$ alors $\alpha_{1}=\tau$, $\alpha_{4}=3-\tau$, $\alpha_{2}=\alpha_{3}=1$. 
\[
\begin{picture}(150,88)
\put(-80,38){$\Gamma(G)$=}
\put(-11,42){\circle{7}}
\put(-8,42){\line(1,0){33}}
\put(29,42){\circle{7}}
\put(32,42){\line(1,0){30}}
\put(65,42){\circle{7}}
\put(68,42){\line(1,0){31}}
\put(103,42){\circle{7}}
\put(107,42){\line(1,0){30}}
\put(141,42){\circle{7}}
\put(-15,52){$s_{1}$}
\put(24,52){$s_{2}$}
\put(61,52){$s_{3}$}
\put(99,52){$s_{4}$}
\put(135,52){$s_{5}$}
\put(4,48){$5$}
\put(44,48){$3$}
\put(80,48){$3$}
\put(116,48){$5$}
\end{picture}
\]
Ici on voit que $G_{1}=<s_{1},s_{2},s_{3},s_{4}>\simeq W(H_{4})$ et on a les mêmes résultats que dans le cas précédent.

\begin{center}
Université de Picardie Jules Verne\\
 Pôle Scientifique\\
Laboratoire LAMFA, UMR CNRS 7352\\
33, rue Saint Leu\\
80039 Amiens Cedex\\
francois.zara@u-picardie.fr
\end{center}
\end{document}